\documentclass[12pt,thmsa, reqno]{amsart}

\usepackage[dvips]{graphics}
\input{amssym.def}
\input{amssym.tex}

\usepackage{lineno}

\def\stnm{{\rm St}(n,m)}

\def\stnk{{\rm St}(n\!+\!1,n\!+\!1\!-\!k)}

\def\sn{\bbs^n}

\def\bbr{{\Bbb R}}

\def\bbh{{\Bbb H}}

\def\bbs{{\Bbb S}}
\def\bbb{{\Bbb B}}

\def\const{{\hbox{\rm const}}}

\def\ker{{\hbox{\rm ker}}}

\def\det{{\hbox{\rm det}}}

\def\rn{\bbr^n}

\def\part{\partial}
\def\intl{\int\limits}
\def\b{\beta}
\def\Lam{\Lambda}

\def\Gam{\Gamma}

\def\a{\alpha}
\def\om{\omega}

\def\del{\delta}
\def\vp{\varphi}

\def\gam{\gamma}
\def\Lam{\Lambda}
\def\sig{\sigma}

\def\z{\zeta}
\def\th{\theta}
\def\e{\varepsilon}
\def\t{\tau}

\def\chi{{\bf 1}}

\def\snm1{\bbs^{n-1}}

\font\frak=eufm10

\def\fr#1{\hbox{\frak #1}}

\def\frM{\fr{M}}

\newtheorem{theorem}{Theorem}[section]
\newtheorem{lemma}[theorem]{Lemma}

\theoremstyle{definition}
\newtheorem{definition}[theorem]{Definition}

\theoremstyle{remark}
\newtheorem{remark}[theorem]{Remark}

\theoremstyle{corollary}
\newtheorem{corollary}[theorem]{Corollary}

\numberwithin{equation}{section}

\newcommand{\be}{\begin{equation}}
\newcommand{\ee}{\end{equation}}

\newcommand{\bea}{\begin{eqnarray}}
\newcommand{\eea}{\end{eqnarray}}
\newcommand{\Bea}{\begin{eqnarray*}}
\newcommand{\Eea}{\end{eqnarray*}}

\def\sideremark#1{\ifvmode\leavevmode\fi\vadjust{\vbox to0pt{\vss
 \hbox to 0pt{\hskip\hsize\hskip1em
\vbox{\hsize2cm\tiny\raggedright\pretolerance10000
 \noindent #1\hfill}\hss}\vbox to8pt{\vfil}\vss}}}%

\begin{document}

\title[ Non-geodesic  Funk Transforms]
{ Non-geodesic Spherical Funk Transforms with One and Two Centers}

\author{M. Agranovsky and B. Rubin}

\address{Department of Mathematics, Bar Ilan University, Ramat-Gan, 5290002, and Holon Institute of Technology, Holon, 5810201, Israel}
\email{agranovs@math.biu.ac.il }

\address{Department of Mathematics, Louisiana State University, Baton Rouge,
Louisiana 70803, USA}
\email{borisr@math.lsu.edu}

\subjclass[2010]{Primary 44A12; Secondary 37E30}



\maketitle

\begin{abstract}  We study non-geodesic Funk-type transforms  on the unit sphere $\sn$ in $\bbr^{n+1}$ associated with cross-sections of $\sn$ by $k$-dimensional planes  passing through an arbitrary  fixed  point inside the sphere. The main results include  injectivity conditions for these transforms, inversion formulas, and   connection with geodesic Funk transforms. We also show that, unlike the case of planes through a single common center, the  integrals over  spherical sections by planes through two distinct  centers provide the corresponding reconstruction problem a unique solution.
 \end{abstract}

\section {Introduction}

Let $\sn$ be the unit sphere in $\bbr^{n+1}$. Given a point  $a$ inside $\sn$, we denote by  ${\rm Gr}_a (n+1, k)$,  $1\le k\le n$,  the Grassmann manifold of  $k$-dimensional affine planes in $\bbr^{n+1}$ passing through  $a$. The aim of the paper is to study  injectivity of the generalized Funk transform
\be\label{gfu} (F_a f)(\t)= \intl_{\sn \cap \t} f(x)\, d\sig (x), \qquad \t \in {\rm Gr}_a (n+1, k),\ee
and obtain inversion formulas for $F_a$ in suitable classes of functions.

  The classical case $F_a=F_o$, when  $a=o$ is the origin, goes back to the pioneering works by Funk  \cite{Fu11, Fu13}  ($n=2$), which were inspired by  Minkowski \cite{Min}. A generalization of the Funk transform $F_o$ to
   arbitrary $1\le k\le n$  is due to Helgason \cite{H90}; see also \cite {H11, Ru02b, Ru15} and references therein. Operators of this kind play an important role in  convex geometry, spherical tomography, and various branches of Analysis \cite{Ga, GGG1, GRS, Ru15, P1, P2}.

   The case  when   $a$ differs from the origin is relatively new in  modern literature, though Funk-type transforms on $\bbs^2$ for noncentral plane sections were considered by   Gindikin,  Reeds, and   Shepp \cite{GRS} in the framework of the kappa-operator theory.
   One should also mention  non-geodesic Funk-type transforms studied by
    Palamodov \cite[Section 5.2]{P2}. Inversion formulas for these transforms were obtained  in terms of delta functions and differential forms. Operators (\ref{gfu}) with $a\neq o$ are  non-geodesic too, however,  they differ from those in  \cite{P2}. In particular, they are  non-injective.  Non-geodesic Funk-type transforms over subspheres of fixed radius were studied by the second co-author in \cite{Ru00a}, where the results fall into the scope of number theory.

 The case $a \neq o$ in (\ref{gfu}) with $k=n$ was considered by Salman; see \cite{Sa16} for $n=2$ and \cite{Sa17} for any $n\ge 2$. To avoid non-uniqueness, he imposed  restriction on the support of the functions that makes his operator different from ours.  The stereographic projection method of \cite{Sa16, Sa17}  makes it possible
to reduce inversion of Salman's operator to the similar problem for a certain Radon-like transform over spheres in $\rn$.

  The next step was due to Quellmalz \cite{Q} for $n=2$, who   expressed $F_a$ through the totally geodesic Funk transform $F_o$ and thus explicitly inverted this operator on a certain  subclass of continuous functions. If $a=o$ this subclass consists of even functions on $\sn$. The results from \cite{Q} were generalized by Quellmalz \cite{Q1} and Rubin \cite{Ru18} to any $n\ge 2$ with $k=n$. The paper \cite{Ru18} also contains an alternative  inversion method for Salman's operators.

   Our aim in the present article is two-fold. First, we   characterize the kernel (the null subspace) of $F_a$ and the subclass of all  continuous functions on which $F_a$ is injective. We also
   obtain  inversion formulas for $F_a$ on that subclass for {\it any}  $1\le k\le n$ and thus generalize the  corresponding results from \cite{Ru18}.

   Second,  to achieve  uniqueness in the reconstruction problem, we consider sections by planes through {\it two distinct centers}. To the best of our knowledge, this approach is new.  We shall prove  that for any pair of distinct  points $a$ and $b$ inside the sphere, the kernels of the corresponding transforms $F_a$ and $F_b$ have  trivial intersection. The latter means that, unlike the case of a single center,  the collection of data from two distinct centers provides  the  reconstruction problem a unique solution.    We also develop an analytic procedure of the reconstruction, that reduces to a certain dynamical system on $\sn$.

   The results of this paper extend to the case  when the point $a$ lies outside $\sn$, and to  arbitrary pairs of  distinct centers $a, b$ in $\bbr^{n+1}$. We plan to address these cases elsewhere.

   \vskip 0.2 truecm

 \noindent {\bf Plan of the Paper.}
 Section 2 contains notation and necessary preliminaries related to  M\"obius-type automorphisms of the sphere.   In Section 3 we   describe the kernel  of  the operator $F_a$ on  continuous functions and characterize the subclass of functions on which  $F_a$ is injective. We also obtain an explicit inversion formula for $F_a$ on that subclass.  Section 4 deals with the system of two equations,  $F_a f=g$, $F_b f=h$, corresponding to distinct centers $a$ and $b$ inside the sphere. Unlike the case of a single common center, such a system determines $f$ uniquely and the function $f$ can be  reconstructed by a certain pointwise convergent  series.  Norm convergence of this series is studied in Section 5. It turns out that the series does not converge uniformly on the entire sphere $\sn$ (only on some compact subsets of $\sn$),
 however, it converges in the $L^p (\sn)$-norm for all $1\le p\le p_0$, $p_0=n/(k-1)$,  and this bound is sharp. In Section 6 we  prove   Theorem  \ref{kdndt}, which was formulated without proof in Section 3. This theorem  plays a key role in the paper. It states   that the shifted  transform $F_a$ is represented  as $F_a=N_a F_o M_a$, where  $N_a$ and $M_a$ are the suitable  bijections and $F_o$ is the classical Funk transform corresponding to $a=o$.

The main results  are contained  in Theorems  \ref{kjy2}, \ref{klut}, \ref{L:uniformly}, and \ref{T:Lp}.


\section {Preliminaries}

\subsection {Notation}

In the following,
$\bbb^{n+1}=\{x\in \bbr^{n+1}: |x|<1\} $ is the open unit ball in $\bbr^{n+1}$, $\sn$ is its boundary,  $x\cdot y$ is the usual dot product.
The notation  $C(\sn)$ and $L^p (\sn)$ for the  corresponding spaces of continuous and $L^p$ functions on  $\sn$   is standard. If $x$ is the variable of integration over $\sn$, then $dx$ stands for the  $O(n+1)$-invariant surface area measure  on $\sn$, so that  $\int_{\sn} dx =2\pi^{(n+1)/2}/\Gam ((n+1)/2)$. We write $d\sig (x)$ for  the induced surface area measure on lower dimensional spherical sections. The letter $x$ can be replaced by another one, depending on the context.

We denote by $\frM_{n,m}$ the space of real matrices having $n$ rows and $m$  columns;
  $\mathrm{M}'$ is the transpose of the matrix $\mathrm{M}$, $\mathrm{I}_m$ is the identity $m \times m$
  matrix.  For $n\geq m$,  $\stnm= \{\mathrm{M} \in \frM_{n,m}: \mathrm{M}'\mathrm{M}=\mathrm{I}_m \}$
 denotes the Stiefel manifold of orthonormal $m$-frames in $\bbr^n$; ${\rm Gr}_a (n, m)$ is the Grassmann manifold of  $m$-dimensional affine planes in $\bbr^{n}$ passing through a fixed point $a$.
 We will be mainly dealing with the  manifolds $\stnk$,  ${\rm Gr}_a (n+1, k)$,  and ${\rm Gr}_o (n+1, k)$ (i.e. $a=o$),  $1\le k\le n$. Given a frame $\xi \in \stnk$, the notation $\xi^\perp $ stands for the $k$-dimensional linear subspace orthogonal to $\xi$; $\{\xi\}$ denotes  an $(n+1-k)$-dimensional linear subspace spanned by $\xi$. All points in $\bbr^{n+1}$ are identified with the corresponding column vectors.

\subsection {Spherical Automorphisms}

We recall some basic facts; see, e.g.,   Rudin
\cite[Section 2.2.1)]{Rud}, Stoll \cite[Section 2.1]{St}. Given a point $a\in \bbb^{n+1}\setminus \{o\}$, we denote by $\mathrm {P}_a$  and $\mathrm {Q}_a=\mathrm {I}_{n+1} -\mathrm {P}_a$  the orthogonal projections of $\bbr^{n+1}$ onto the direction of $a$ and  the subspace $a^\perp$, respectively.  If $x\in  \bbr^{n+1}$, then
 \[\mathrm {P}_a  x = \frac {a\cdot x}{|a|^2}\,a.\]
Let
 \be\label{E:g_a}
\vp_a x=\frac{a-\mathrm {P}_a x - s_a \mathrm {Q}_a x }{1-x\cdot a}, \qquad s_a=\sqrt{1-|a|^2},\ee
which  is a one-to-one M\"obius transformation satisfying
\be\label{Evf} \vp_a (o)=a, \qquad \vp_a (a)=o, \qquad  \vp_a (\vp_a x)=x,\ee
\be\label{Evf1}
1-|\vp_a x|^2=\frac{(1-|a|^2)(1-|x|^2)}{(1-x\cdot a)^2}, \qquad x\cdot a \neq 1.\ee
If $x\in \sn$, then
\be\label{wqvE:g_a5}
\frac{1-a\cdot \vp_a x}{1+ a\cdot \vp_a x}=\frac{1-|a|^2}{|a-x|^2}.\ee
Properties (\ref {Evf})-(\ref{Evf1}) can be checked by straightforward computation. By (\ref{Evf1}), $\vp_a$  maps the ball $ \bbb^{n+1}$ onto itself and  preserves $\sn$.

\begin{remark}
It is known that the ball $ \bbb^{n+1}$ with the relevant metric can be considered as the Poincar\'e model of the real $(n+1)$-dimensional hyperbolic space $\bbh^{n+1}$.  There is an intimate connection between the M\"obius transformations of $ \bbb^{n+1}$ and the
 group $O(1, n+1)$  in  the hyperboloid model of $\bbh^{n+1}$. In the present article we do not exploit this connection. An interested reader may be referred, e.g., to Beardon \cite [Section 3.7]{B},
   Gehring,  Martin,  Palka \cite [Section 3.7] {GMP},
   Mostow \cite[Theorem 1.1]{Mo}.
\end{remark}

\begin{lemma} \label{lem31}  For any $f\in L^1 (\sn)$,
\be\label{fff}\intl_{\sn} f(x)\, dx= s_a^n \intl_{\sn} \frac{(f\circ \vp_a)(y)}
{(1-  a \cdot y)^n}\, dy, \qquad s_a=\sqrt{1-|a|^2}.\ee
\end{lemma}
\begin{proof} We write $x$  in spherical coordinates
\[ x =\sqrt {1-u^2} \, \th +u \tilde a, \qquad \tilde a=\frac{a}{|a|}, \quad |u|\le 1, \quad \th \in \sn \cap a^\perp,\]
to obtain
\be \label{cnhos} \intl_{\sn} \!f(x)\, dx\!= \!\intl_{-1}^1  (1\!-\!u^2)^{(n-2)/2}\,
 du\!\intl_{\sn \cap \,a^\perp} \!\!f \left( \sqrt{1\!-\!u^2}\, \th \!+\!u\, \tilde a\right)\,d\th.\ee
 By (\ref{E:g_a}),
\be\label{ouq} \vp_a x= -\sqrt {1-v^2} \, \th +v \tilde a, \qquad v =\frac{|a|-u}{1-|a|u}.\ee
 Note that the map $u\to v$ is an involution. Changing
  variable
 \[ u=\frac{|a|-v}{1-|a|v}\]
and taking into account that
 \[ \frac{du}{dv}=\frac{|a|^2-1}{(1-|a|v)^2}, \qquad 1-u^2= \frac{(1-|a|^2)(1-v^2)}{(1-|a|v)^2}, \]
we have
\bea
\intl_{\sn} f(x)\, dx &=&
(1-|a|^2)^{n/2} \intl_{-1}^1  \frac{(1\!-\!v^2)^{(n-2)/2}}{(1-|a|v)^n}\, dv \nonumber\\
&\times& \intl_{\sn \cap \,a^\perp}  \!\!f \left(  \frac{\sqrt{1-|a|^2} \, \sqrt{1-v^2}}{1-|a|v}\, \th +\frac{|a|-v}{1-|a|v}\,\tilde a\right )\, d\th\nonumber\\
 &=&
 s_a^n  \intl_{\sn} \frac{(f\circ \vp_a)(y)}{(1-  a \cdot y)^n} \, dy, \quad \text{\rm as desired}.\nonumber\eea
\end{proof}

We also define the reflection  $\tau_a: \,\sn \to \sn$ about the point $a\in \bbb^{n+1}$:
\be\label {lab2in} \tau_a x=\frac{(|a|^2 -1)\, x + 2(1-x\cdot a)\, a }{|x -a|^2}.\ee
It assigns to $x\in \sn$ the antipodal point $\tau_a x \in \sn$ that lies on the line passing through $x$ and $a$. A similar reflection map about the origin $o$ is denoted by $\tau_o$, so that  $\tau_o x=-x$.

The map $\vp_a$ intertwines  reflections $\tau_a$ and $ \tau_o$, that is,
\be\label{vE:g_a5}
 \vp_a \tau_a=\tau_o \vp_a .\ee
Indeed, $\vp_a$ maps  chords of the ball onto chords. Hence,
for any $ x\in \sn$,  the segment $[x,\tau_a x]$ is mapped onto the segment
$[\vp_a x, \vp_a\tau_a x]$. Since the first segment contains $a$, the second one contains $\vp_a(a)=o$. The latter means that the
points $\vp_a x$ and $ \vp_a\tau_a x $  are
symmetric with respect  to the origin, that is, $\vp_a \tau_a x=\tau_o \vp_a x$.

\begin{lemma} \label{lem312}  If $f\in L^1 (\sn)$ and  $a\in \bbb^{n+1}$, then
\be\label{fffza}\intl_{\sn} f(\tau_a x)\, dx= \intl_{\sn} f(x)\, \left(\frac{1-|a|^2}{|a-x|^{2}}\right )^n\, dx,
\ee
\be\label{fffzas}\intl_{\sn} f(x)\, dx= \intl_{\sn} f(\tau_a x)\, \left(\frac{1-|a|^2}{|a-x|^{2}}\right )^n\, dx.
\ee
\end{lemma}
\begin{proof}  By (\ref{vE:g_a5}) and (\ref{fff}),
\bea \intl_{\sn} f(\tau_a x)\, dx &=&\intl_{\sn} f(\vp_a \tau_o \vp_a x)\, dx \qquad \text {\rm (set $x= \vp_a \tau_o y$)}\nonumber\\
&=&(1-|a|^2)^{n/2}\intl_{\sn}  \frac{(f\circ \vp_a)(y)}
{(1-a \cdot y)^n}\,   \left (\frac{1-a \cdot y}
{1+a \cdot y}\right )^n\,  dy\nonumber\\
 &=&\intl_{\sn} f(x)\,\left (\frac{1-a \cdot \vp_a x}
{1+a \cdot \vp_a x}\right )^n\,  dx.\nonumber\eea
It remains to apply (\ref{wqvE:g_a5}).
The second equality follows from the first one: just replace  $f(x)$ by  $f(\tau_a x)$ and use $\tau_a \tau_a x =x$.
\end{proof}

\section{The Shifted Funk Transform}

\subsection {Inversion Procedure}  The  following theorem establishes connection between
 the shifted  Funk transform
  \be\label{fual} (F_a f)(\t) = \intl_{\sn \cap \,\t} f(x)\, d\sig (x), \qquad \t \in {\rm Gr}_a (n+1, k), \ee
  and the classical  Funk transform $F_o=F_a\vert_{a=o}$ that takes functions on $\sn$ to functions on  ${\rm Gr}_o (n+1, k)$. Given a function $f$ on $\sn$ and a function $\Phi$ on  ${\rm Gr}_o (n+1, k)$, we denote
\be\label{vt1de}
(M_af)(y)\!=\!\left ( \frac {s_a}{1\!-\!a\cdot  y}\right )^{k-1}\!\!(f\circ \vp_a)(y), \quad (N_a\Phi)(\t)= \Phi (\vp_a \t), \ee
where $s_a=\sqrt{1-|a|^2}$ and $\vp_a$ is an automorphism (\ref{E:g_a}).
\begin{theorem}\label {kdndt} Let $1\le k\le n$, $a\in \bbb^{n+1}$.
  If  $f\in C(\sn)$,   then
\be\label{vt1ata}
F_af=N_a F_o M_a f. \ee
\end{theorem}

The  proof of  this theorem is given in Section \ref{jsbsr}.

The  Funk transform $F_o$   is injective on the subspace $C^{+} (\sn)$ of even functions, whilst the subspace $C^{-} (\sn)$ of  odd functions  is the kernel of $F_o$ in $C(\sn)$; see, e.g.,   \cite {H11, Ru02b, Ru13b, Ru15}. We denote by $\tilde F_o$ the restriction of $F_o$ onto  $C^{+} (\sn)$.

There exist several
different approaches to inversion of $\tilde F_o$. We recall one of them.  Given   $\vp=\tilde F_o f$,  $f\in C^{+} (\sn) $, consider the mean value operator
\be\label{liuh4}
(F_x^* \vp)(r)=\intl_{\{\z  \in {\rm Gr}_o (n+1, k): \,d(x,\z)=r\}} \vp(\z)\, d m (\z), \qquad 0<r<1,\ee
where integration is performed with respect to the relevant probability measure over the set of all planes
$\z  \in {\rm Gr}_o (n+1, k)$ at geodesic distance $d(x,\z)=\cos^{-1} r$ from $x$.

\begin{theorem}\label{invrhys} {\rm (cf. \cite[Theorem 5.3]{Ru13b})}
A function $f\in C^{+} (\sn)$  can be reconstructed from $\vp=\tilde F_o f$ by
\bea\label{90ashel}
f(x)&\equiv&(\tilde F_o^{-1} \vp)(x) \\
&=&\lim\limits_{s\to 1}  \left (\frac {1}{2s}\,\frac {\partial}{\partial s}\right )^k \Bigg [\frac{\pi^{-k/2}}{\Gam (k/2)}\intl_0^s
(s^2 \!- \!r^2)^{k/2-1} \,(F^*_{x} \vp)(r) \,r^k\,dr\Bigg ].\nonumber\eea
In particular, for $k$ even,
\be\label{90ashele}
(\tilde F_o^{-1} \vp)(x) \!=  \! \lim\limits_{s\to 1} \frac{1}{2\pi^{k/2}} \left (\frac {1}{2s}\,\frac {\partial}{\partial s}\right )^{k/2}[s^{k-1}(F^*_{x} \vp) (s)].\ee
The  limit in these formulas is understood in the  $\sup$-norm.
\end {theorem}

Now we proceed to inversion of $F_a$, which, by Theorem \ref{kdndt}, is factorized as $F_a=N_a F_o M_a$.  Here the operators $M_a$ and $N_a$  are injective, so that
\be\label{tyhce1} (M_{a}^{-1}  f)(x)=(1-a\cdot \vp_a x)^{k-1}\, (f \circ \vp_a)(x), \qquad N_{a}^{-1} \Phi= \Phi\circ \vp_a.\ee
The following definition is motivated by the factorization $F_a=N_a F_o M_a$ and nicely agrees with the case $a=o$.

\begin{definition} \label {uqg33} {\it  A  function $f\in C(\sn)$  is called $a$-even (or  $a$-odd) if
$M_a f$  is even (or odd, resp.) in the usual sense.
 The subspaces of all $a$-even and $a$-odd continuous functions on $\sn$ will be denoted by $C_a^+ (\sn)$ and $C_a^- (\sn)$, respectively.  The restriction  of $F_a$ onto   $C_a^{+}(\sn)$ will be denoted by  $\tilde F_a$.}
 \end{definition}

\begin{theorem} \label {kjy2} Let $1<k\le n$. Then
 $ \ker \,(F_a)=C_a^{-}(\sn)$ and the restricted operator   $\tilde F_a$  is injective.
 A  function $f\in C_a^{+}(\sn) $ can be uniquely reconstructed from $g=\tilde F_a f$ by
 \be\label{hhhht1}
f\equiv \tilde F^{-1}_a g=M_{a}^{-1} \tilde F_o^{-1} N_{a}^{-1} g,\ee
where $M_{a}^{-1}$, $ \tilde F_o^{-1}$, and $N_{a}^{-1}$ are defined by  (\ref{tyhce1}) and Theorem \ref{invrhys}.
   \end {theorem}

This statement is an immediate consequence of (\ref{vt1ata}) and the corresponding results for $F_o$.

\begin{remark} \label {klut} In the  case $k=1$, which is not included in Theorem \ref {kjy2}, the plane $\t$ is a line and  the integral (\ref{gfu}) is a sum of the values of $f$ at the  points where this line intersects the sphere. If $ x$ is one of such points  and $L_{a,x}$ is the line through $a$ and $x$, then
\be\label {l762}
(F_a f)(L_{a,x}) = f(x) + f(\tau_a x).\ee
The  $a$-odd functions, for which  $f(x)=-f(\tau_a x)$,  form the kernel of the operator (\ref{l762}). An  $a$-even  function $f$,  satisfying $f(x)=f(\tau_a x)$, can be reconstructed from $(F_a f)(L_{a,x})$  by the formula
\be\label {l762a} f(x) =\frac{1}{2} \, (F_a f)(L_{a,x}).\ee
\end{remark}

\subsection {Alternative description of the  subspaces $C_a^{\pm}(\sn)$}
 We set
\be\label{jytr0}
\rho_a(x)=\left(\frac{1-|a|^2}{|a-x|^2}\right)^{k-1}, \qquad (W_{a} f)(x)=\rho_a(x) f (\tau_a x),\ee
where  $\tau_a$ is the  reflection (\ref{lab2in}).
\begin{lemma} \label {u54aq} The operator $W_{a}$ is an involution, i.e., $W_{a}W_{a} f=f$.
\end{lemma}
\begin{proof} The  statement is obvious for $a=o$, when  $(W_{0} f)(x)=f(-x)$. It is also obvious for any $a\in \bbb^n$ if $k=1$.
In the general case, taking into account that $\tau_a\tau_a x=x$,
 we have
\[(W_{a}W_{a} f)(x)=\left [\frac{1-|a|^2}{|a-x|^2}\,  \frac{1-|a|^2}{|a-\tau_a x|^2}\right ]^{k-1} \, f(x).\]
 By  (\ref {wqvE:g_a5}) and  (\ref{vE:g_a5}), the expression in square brackets can be written as
\[
\frac{(1-a\cdot \vp_a x)\, (1-a\cdot \vp_a \tau_a x)}{(1+ a\cdot \vp_a x)\, (1+ a\cdot \vp_a \tau_a x)}  =
\frac{(1-a\cdot \vp_a x)\, (1+ a\cdot \vp_a x) }{(1+ a\cdot \vp_a x)\, (1-a\cdot \vp_a x)} =1.\]
This gives the result.
\end{proof}

\begin{theorem} \label {akjy2} A function  $f\in C(\sn)$  is $a$-even (or $a$-odd) if and only if $f=W_af$ (or $f=-W_af$, respectively).
  \end {theorem}
\begin{proof}  By  Definition \ref{uqg33},  $f\in C(\sn)$  is $a$-even if and only if
 $(M_{a} f)(y)= (M_{a} f)(-y)$ for all $y \in \sn$. The latter is equivalent to
\[ (f \circ \vp_a)(y)= \left (\frac{1-a\cdot y}{{1+a\cdot y}}\right )^{k-1} (f \circ \vp_a)(-y),\]
or (set $y=\vp_a x$ and use (\ref{wqvE:g_a5}) and (\ref{vE:g_a5}))
\[
f(x) = \left (\frac{1-a\cdot \vp_a x}{1+ a\cdot \vp_a x}\right )^{k-1}  f(\tau_a x)= \rho_a(x) f (\tau_a x)=(W_af)(x).\]
 The proof for the $a$-odd functions is similar.
\end{proof}

\begin{corollary} Every function $f\in C(\sn)$ can be represented as a sum of its $a$-even and $a$-odd parts. Specifically,
\be\label {klutv} f=f_a^+ +  f_a^-, \qquad  f_a^{\pm}  =\frac{f \pm W_af}{2}.\ee
\end{corollary}
\begin{proof} The first equality follows from the second one. Further, by Lemma \ref{u54aq},
\[W_ a f_a^{\pm} = \frac{W_a f  \pm W_aW_a f }{2} = \frac{W_a f  \pm  f}{2} = \pm f_a^{\pm}.\]
 Hence, by Theorem \ref{akjy2}, $f_a^+$ is $a$-even and  $f_a^-$ is $a$-odd.
\end{proof}

\section{Reconstruction from Two Centers}

 As we have seen in Section 3, a generic function $f\in C(\sn)$ cannot be reconstructed from  $F_a f$. Because  $F_a f=\tilde F_a f_a^+$,
   one can reconstruct only  the $a$-even part $f_a^+$ of $f$, whilst the $a$-odd part  is lost.
Our aim  is to show that complete reconstruction becomes possible if we
consider  two distinct centers  instead of one. Specifically, let $a, b\in \bbb^{n+1}$, $a\neq b$. Consider the
 system of two equations
\be\label {aaa1}F_a f=g, \qquad F_b f=h,\ee
and suppose that a function $f\in C(\sn)$ satisfies  this system. Then   $\tilde F_a f_a^+ =g$,  $\tilde F_b f_b^+=h$, and therefore, by (\ref{klutv}),
\[ f_a^{+}  \equiv \frac{f + W_af}{2}=\tilde F_a^{-1} g, \qquad  f_b^{+}  \equiv \frac{f + W_bf}{2}=\tilde F_b^{-1}h,\]
where
\be\label {ksa2} (W_{a} f)(x)=\rho_a(x) f (\tau_a x), \qquad (W_{b} f)(x)=\rho_b(x) f (\tau_b x).\ee
 Setting
\[ g_1= 2\tilde F_a^{-1} g, \qquad h_1= 2\tilde F_b^{-1}h,\]
 we obtain a pair of functional equations
\be\label {adaa1}  f= g_1 -W_{a} f, \qquad  f=h_1 -W_{b} f. \ee
 Then we substitute $f$ from the second equation into the right-hand side of the first one to get
\be\label {aaa1s} f =W f +q, \qquad W=W_{a}W_{b}, \qquad q= g_1-W_a h_1.\ee
Iterating (\ref{aaa1s}), we obtain
\begin{equation}\label{E:f=sum}
f=W^m f + \sum\limits_{j=0}^{m-1} W^j q; \qquad m=1,2, \ldots \, .
\end{equation}

This equation generates a dynamical system on $\sn$.

\begin{lemma}\label{lug6n} Let $a^*$ and $ b^* $ be the points  on $\sn$ that lie on the straight line through $a$ and $b$. Suppose that $a$ is closer to $a^*$ than $b$. If $W=W_{a}W_{b}$, then $\lim\limits_{m\to \infty} (W^m f)(x) =0$ for all $x\in \sn \setminus \{a^*\}$  and  $1< k \leq n$. If $k=1$ and   $x \in \sn \setminus \{ a^*\}$, then
$\lim\limits_{m \to  \infty}(W^mf)(x)=f(b^*)$.
\end{lemma}
\begin{proof}
We observe that
\be\label{sb5}
(W f)(x)=(W_{a}W_{b} f)(x)=\rho_a(x)\rho_b(\tau_a x) f(\tau_b\tau_a x).\ee
Denote
\be\label{srr2a} \rho (x)\!=\!\rho_a (x) \rho_b (\tau_a x)\!=\!\left[ \frac{(1-|a|^2)\, (1-|b|^2)}{|a-x|^2 \, |b-\tau_a x|^2}\right]^{k-1}, \quad
\mathrm{T}= \tau_b \tau_a.\ee
Then $(W f)(x) = \rho (x) f(\mathrm{T}x)$ and, by iteration,
\be\label{sfr2c5}(W^m f)(x) =  \om_m(x)\,f (\mathrm{T}^{m+1} x), \qquad  \om_m(x)= \prod\limits_{j=0}^{m} \rho (\mathrm{T}^j x).\ee
For any $x\neq a^*$, the mapping $\mathrm{T}$ preserves the circle $C_{x,a,b}$ in the 2-plane spanned by $x$, $a$ and $b$, and leaves the  points $a^*$ and $b^*$ fixed. A simple geometric consideration in the 2-plane shows  that  the distance from the points $\mathrm{T}^j x\in C_{x,a,b}$  to  $b^*$ monotonically decreases, and therefore, the sequence $\mathrm{T}^j x$ has a limit. This limit must be a fixed point of the mapping $\mathrm{T}$, and hence $\mathrm{T}^j x \to b^*$ as $j \to \infty$. Because $\rho$ is continuous, it follows that
 \be\label{stttgy} \lim\limits_{j\to \infty} \rho (\mathrm{T}^j x) =\rho (b^*).\ee
Using this fact, let us show that if $k>1$, then
\be\label{sfr2c5b}
\lim\limits_{m\to \infty} \om_m(x)=0.\ee
Once (\ref{sfr2c5b}) has been  proved, the statement of the lemma for $k>1$ will follow because the factor $f(\mathrm{T}^{m+1}x)$ has  finite limit $f(b^*)$.

To prove (\ref{sfr2c5b}), it suffices to show that
\be\label{sgvf21}
\rho (b^*) <1,\ee
where, by (\ref{srr2a}),
\be\label{sgvf21e}
\rho(b^*)=\left [ \frac{(1-|a|^2)\,(1-|b|^2)}{|a-b^*|^2\, |a^* -b|^2} \right ]^{k-1}.\ee
Let
\be\label{sgvy1e} a=a^*\!+t(b^* \!-\!a^*), \quad b =a^*\!+s(b^* \!-\!a^*), \qquad 0 < t <s <1.\ee
Taking into account that $|a^*|=|b^*|=1$ and using (\ref{sgvy1e}), we obtain
\be\label{asgvy1e}
1\!-\!|a|^2= 2t(1\!-\!t)(1\!-\!a^* \!\cdot b^*), \qquad 1\!-\!|b|^2= 2s(1\!-\!s)(1\!-\!a^* \!\cdot b^*), \ee
\[
|a\!-\!b^*|^2=2(1\!-\!t)^2 (1\!-\!a^* \!\cdot b^*), \qquad  |a^*\!- \!b|^2=2s^2 (1\!-\!a^* \!\cdot b^*).\]
Hence
\be\label{suu1e}
\rho(b^*)=\left[ \frac{t(1-s)}{s(1-t)} \right ]^{k-1}<1.\ee
 The last inequality is an immediate consequence of the assumption $0 < t <s <1$.

The case $k=1$ is simpler. In this case $\rho(x)=1$, and therefore, $(W^mf)(x)=f(\mathrm{T}^{m+1}x) \to f(b^*)$ as $m \to \infty$, $x \in \sn \setminus \{ a^*\}.$
\end{proof}

The above reasoning  yields the following  preliminary conclusion.
If $a \neq b$, then, by Theorem \ref{kjy2},  the kernel of the map $f \to (F_af,F_bf)$, $f\in C(\sn)$, is $C_a^-(\sn)\cap C_b^-(\sn)$.  But if $f$ is odd with respect to both $a$ and $b$,  then, by Theorem \ref{akjy2}, $Wf=f$.   By  Lemma \ref{lug6n}  it  follows  that $f(x)  =  0$  for all $x \in \sn \setminus \{a^*\}$.  However, since $f$ is continuous, we must have $f= 0$ {\it everywhere} on $\sn$.  In particular, it follows that $f$ can be reconstructed from the knowledge of $F_af$ and $F_bf$,  or, what is the same, from $f^+_a$ and $f^+_b$.

More precisely, we have  the following result.

\begin{theorem}\label {klut}  Let $W_{a}$ and $W_{b}$ be  involutions (\ref {ksa2}), $1<k\le n$. If  the system of equations $F_a f=g$ and $F_b f=h$ has a solution $f\in C(\sn)$, then this solution  is unique and can defined by the pointwise convergent series
\be\label {bbi}
f(x)=\sum\limits_{j=0}^{\infty} W^j q(x), \qquad x\neq a^*,\ee
where $ W= W_{a}W_{b}$, $q=2\, [\tilde F_a^{-1} g - W_a \tilde F_b^{-1}h]$,
$\tilde F_a^{-1}$ and $\tilde F_b^{-1}$ being defined as in (\ref{hhhht1}).
Alternatively,
\be\label {bbi2}
f(x)=\sum\limits_{j=0}^{\infty} \tilde W^j r(x), \qquad x\neq b^*,\ee
where $ \tilde W= W_{b}W_{a}$ and $r=2\, [\tilde F_b^{-1} h - W_b \tilde F_a^{-1}g]$.
\end{theorem}
\begin{proof}
To prove (\ref{bbi}),  it suffices to pass to the limit in (\ref{E:f=sum}), taking into account that, by Lemma \ref{lug6n},
the remainder $(W^m f)(x)$ of the series (\ref{bbi})  converges to zero
for every  $x\neq a^*$.  An alternative formula (\ref{bbi2}) then follows if we  interchange  $a$ and $b$, $g$ and $h$.
\end{proof}

\begin{remark}
In the case  $k=1$,
 a function $f\in C(\sn)$ can be reconstructed from the system $F_a f=g$, $ F_b f=h$ as follows.
 By Lemma \ref{lug6n}, $(W^m f)(x) \to f(b^*)$ as $m \to \infty$. Hence
\be
\label{kjyyr} f(x)=\sum\limits_{j=0}^{\infty} q(\mathrm{T}^{j+1}x)+f(b^*), \qquad x \neq a^*, \quad \mathrm{T}=\tau_b \tau_a,\ee
where $q(x)=2\, [(\tilde F_a^{-1} g)(x) - (\tilde F_b^{-1}h)(\tau_a x)]$. By (\ref{l762a}),
 \[
(\tilde F_a^{-1} g)(x)= \frac{1}{2} \,g(L_{a,x}), \qquad (\tilde F_b^{-1}h)(\tau_a x)=\frac{1}{2} \,h(L_{b,\tau_a x}),\]
where the line $L_{a,x}$  passes through $a$ and $x$ and     $L_{b,\tau_a x}$ passes  through $b$ and $\tau_a x$.  It follows that
\be\label {luki91} q(x)=g(L_{a,x}) - h(L_{b,\tau_a x}).\ee
Similarly, $(\tilde W^m f)(x) \to f(a^*)$, and we have
\be
\label{kjyyr1} f(x)=\sum\limits_{j=0}^{\infty} r(\tilde {\mathrm{T}}^{j+1}x) +f(a^*), \qquad x \neq b^*, \quad \tilde {\mathrm{T}}=\tau_a \tau_b,\ee
\be\label {luki91z} r(x)= h(L_{b,x}) - g(L_{a,\tau_b x}).\ee

The series (\ref{kjyyr})  and (\ref{kjyyr1}) reconstruct $f$ up to  unknown additive constants $f(a^*)$ or  $f(b^*)$, where $a^*$  and $b^*$ are the endpoints of the chord through $a$ and $b$.
However, complete reconstruction is still possible, if we  apply symmetrization, by summing (\ref{kjyyr})  and (\ref{kjyyr1}). This gives the following result.

\begin{theorem}\label {klut1} Let $k=1$. Then
\begin{equation} \label{E:bbi3}
2 f(x)= \sum\limits_{j=0}^{\infty} q(\mathrm{T}^{j+1}x)+ \sum\limits_{j=0}^{\infty} r(\tilde {\mathrm{T}}^{j+1}x) + F_a(L_{a,b}), \qquad x \neq a^*, b^*,
\end{equation}
where $q$ and $r$ are defined by (\ref{luki91}) and (\ref{luki91z}), respectively, $L_{a,b}$ is the line  through $a$ and $b$, and $F_a(L_{a,b})=f(a^*)+f(b^*)  \,  (=F_b(L_{a,b}))$ is known.  The values of $f$ at the  points $a^*$ and $b^*$ can be reconstructed by continuity.
\end{theorem}
\end{remark}

 \section{ Norm Convergence  of the Reconstructing Series }

Reconstruction of $f$ by the pointwise convergent series (\ref{bbi}) and  (\ref{bbi2}) gives  a little
possibility to control  the accuracy of the result because the rate of the pointwise convergence depends on the point. Therefore, it is natural to look at the convergence in certain normed spaces. Below  we explore such convergence  in the spaces $C(\sn)$ and $L^p(\sn)$.  As above, we keep the notation  $a^*$  and $b^*$ for the endpoints of the chord through $a$ and $b$.

Consider the most interesting case $k>1$. By (\ref{E:f=sum}), the convergence of the series (\ref{bbi}) to $f$  is equivalent to convergence of its remainder $W^mf$ to $0$ as $m\to \infty$. Thus, it suffices to  confine  to  $W^mf$.

We first note that
the  series (\ref{bbi}) may diverge at the point $a^*$. Indeed, because
$(W^mf)(a^*)=f(\mathrm{T}^{m+1}a^*)\prod\limits_{j=0}^m \rho(\mathrm{T}^j a^*)$ and
 $a^*$ is a fixed point of the mapping $\mathrm{T}$, we have
$$(W^m f)(a^*)= \rho (a^*)^{m+1} f(a^*), \qquad \rho (a^*)=\left [\frac{(1-|a|^2)(1-|b|^2)}{|a-a^*|^2|b-b^*|^2}\right ]^{k-1}.$$
Suppose that $a$ and $b$  are symmetric with respect to the origin and
 $|a|=|b|=1/2$.  Then
\[
 \rho (a^*)\!=\! \left [\frac{(1+|a|)(1+|b|)}{(1-|a|)(1-|b|)}\right ]^{k-1}=9^{k-1},\]
and therefore $(W^m f)(a^*)=9^{(k-1)(m+1)} f(a^*) \to \infty$ as  $m \to \infty$ whenever $f(a^*) \neq 0$.
The latter means that  if $f(a^*) \neq 0$, then the series (\ref{bbi})  diverges at $a^*$  and  its uniform convergence  on the entire sphere  fails. Below it  will be shown that  the uniform convergence of this series fails for any $a, b \in \bbb^{n+1}$.

To understand the type of convergence, we need  a deeper insight in the dynamics of involved reflections.

\subsection{Dynamics of the Double Reflection Mapping $\mathrm{T}=\tau_b\tau_a$}
We know that the trajectory $\{\mathrm{T}^m x: \, m=0,1,2,\ldots \}$ of any point $ x \in \sn \setminus \{a^*\}$ converges to the point $b^*,$ which is the endpoint of the chord containing $a$ and $b$. Let us specify the character of this convergence.

\begin{lemma}\label{L:attracting} The mapping $\mathrm{T}=\tau_b\tau_a$ maps the punctured sphere $\sn_a=\sn \setminus \{a^*\}$ onto itself.
The point $b^*$ is the attracting point of the dynamical system $T^m: \,\sn_a \to \sn_a$ uniformly on compact subsets, that is, for any open neighborhood $U \subset \sn_a$ of  $b^*$  and any compact set $K \subset \sn_a$ there exists $\overline m$ such that $T^m K \subset U$ for all $m \geq \overline m$.
 \end{lemma}
\begin{proof} The first statement is obvious, because $\mathrm{T}a^*=a^*$ and $\mathrm{T}^{-1}a^*=\tau_a\tau_b a^*=a^*$.
The second statement follows by a standard argument for monotone pointwise convergence  on compacts.
In fact, it suffices to prove this statement for the sets $U$ and $K$  having  the form
\[
U=U_\e=B(b^*, \varepsilon), \qquad  K=K_{\delta}=\sn \setminus B(a^*,\delta),
\]
where $B(a^*,\varepsilon)$ and $B(b^*,\delta)$ are  geodesic balls in $\sn$ of sufficiently small radii.

The pointwise convergence yields that for any fixed $x_0 \in K_{\delta}$ there exists a number $m_{0}$ such
that $\mathrm{T}^{m_0} x_0 \in U_\e$. By continuity, the same is true for every $x$ in some neighborhood
$V_{x_0}$ of $x_0$. Thus, the compact $K_{\delta}$ is covered by open sets $V_x$,  $x \in K_{\delta}$, and
therefore we can cover $K_{\delta}$
by a finite family  $\{V_{x_1}, \ldots  V_{x_M}\}$. For each $x_i$, there is a number $m_i$ such that $\mathrm{T}^{m_i} x_i \in U_\e$.
Setting
$ \overline m=\max\{m_{1},...,m_{M}\}$, we have
$$
\mathrm{T}^{\overline m} K_{\delta} \subset U_\e.
$$
 A simple geometric consideration shows that the sequence  $\mathrm{T}^{m+1}K_{\delta}$ monotonically decreases, i.e.,  $\mathrm{T}^{m+1}K_{\delta} \subset \mathrm{T}^m K_{\delta}$. Hence $\mathrm{T}^m K_{\delta} \subset U_\e$  for all $m \geq \overline m$.
 \end{proof}

\subsection{ Uniform Convergence on Compact Subsets of the Punctured Sphere}

\begin{theorem} \label{L:uniformly} If $f \in C(\sn)$, then  the series (\ref{bbi}) converges to $f$ uniformly on compact subsets of the punctured sphere $\sn \setminus \{a^*\}$.
\end{theorem}
\begin{proof} Consider the remainder $(W^m f)(x)$ of the series (\ref{bbi}). By (\ref{sfr2c5}),
\[(W^m f)(x) =  \om_m(x)\,f (\mathrm{T}^{m+1} x), \qquad  \om_m(x)= \prod\limits_{j=0}^{m} \rho (\mathrm{T}^j x).\]
  Because $\rho(b^*) <1$ (see (\ref{suu1e})), for a fixed $\gam $ satisfying  $\rho(b^*) < \gam <1$, there is an open neighborhood $U \subset \sn \setminus \{a^*\}$ of the point $b^*$ such that $0< \rho(y)< \gam$ for all $y \in U$.
On the other hand, Lemma \ref{L:attracting} says
 that there exists $\overline m$ such that $\mathrm{T}^m K_{\delta} \subset U$ for $ m \geq \overline m$ and hence $0< \rho(\mathrm{T}^m x) < \gam$ for all $ x \in K_{\delta}$ and $ m \geq \overline m$. Thus
 $$
\om_m(x)  \leq  \gam^{m-\overline m}\, \max\limits_{x \in K_{\delta}} \prod\limits_{j=0}^{\overline m}\rho(\mathrm{T}^j x)
 $$
for all  $m \geq \overline m$ and  all $x \in K_{\delta}$. It follows that $\om_m(x) \to 0$  as $m \to \infty$ uniformly on $K_{\delta}$.
Since $|f(T^mx)| \leq \|f\|_{C(S^n)}$, we conclude that $W^m f \to 0$ uniformly on $K_{\delta}$. This gives the result.
\end{proof}

\subsection{$L^p$-Convergence}

\begin{lemma} \label{L:isometry}
The operators $W_a$,  $W_b$, $W=W_a W_b$, and $\tilde W=W_b W_a$ are isometries of the space $L^{p_0}(\sn)$ with $p_0=n/(k-1)$.
\end{lemma}
\begin{proof} The statement about $W_a$  follows from (\ref{fffzas}), which reads
$$
\int\limits_{\sn} \big( \rho_a(x) \big)^{p_0}f (\tau_a x) dx= \int\limits_{\sn} f(x) \,dx.
$$
The equality holds for any $f \in L^1(\sn)$ and therefore, if $f \in L^{p_0}(\sn)$, then, using   $|f(x)|^{p_0}$ instead of $f$, we obtain
$\|W_a f\|_{p_0}=\|f\|_{p_0}$. The statement for $W_b$ follows analogously. The operators $W$ and $\tilde W$ are also isometries, as the products of two isometries.
\end{proof}

\begin{theorem}\label{T:Lp} Let $f \in C(S^n)$, $p_0=n/(k-1)$.  The series (\ref{bbi}) and (\ref{bbi2}) converge to $f$  in the norm of $L^p(\sn)$ for any $1 \leq  p < p_0$.
The convergence to $f$ fails in any space $L^p(\sn)$ with $p_0 \leq p \leq \infty.$
\end{theorem}
\begin{proof} It is clear that  $f$  belongs to $L^p(\sn)$ for any $1 \leq p \leq \infty$.  Fix $\delta >0$ and consider the  function
$W^m f=(W_a W_b)^m f$.  Suppose that $p<p_0$ and set $r=p_0/p>1$. We write
\bea \label{E:I12}
\|W^mf \|^p_p &=&\int\limits_{B(a^*,\delta)}|(W^m f)(x)|^p \,dx + \int\limits_{K_{\delta}}|(W^mf)(x)|^p dx \nonumber\\
\label{E:I12} &=&I_1(m,\delta)+I_2(m,\delta), \eea
where, as above, $K_{\delta}=\sn \setminus B(a^*,\delta)$. By H\"older's inequality,
$$
I_1(m,\delta) \leq  \Bigg( \int\limits_{B(a^*,\delta)} \Big ( |(W^mf)(x)|^p  \Big)^{r}dx  \Bigg)^{1/r}
\Bigg( \int\limits_{B(a^*,\delta)} dx \Bigg)^{r/(r-1)}.
$$
Owing to Lemma \ref{L:isometry}, the operator $W^m$ preserves the $L^{p_0}$-norm, and therefore
$$
\begin{aligned}
&\Bigg( \int\limits_{B(a^*,\delta)} \Big (|(W^mf)(x)|^p \Big)^r dx\Bigg)^{1/r} = \Bigg( \int\limits_{B(a^*,\delta)} |(W^mf)(x)|^{p_0}dx \Bigg)^{1/r} \\
&\leq  \| W^m f \|_{p_0}^{p_0/r}=\|f\|_{p_0}^p.
\end{aligned}
$$
Hence
\begin{equation}\label{E:I1}
I_1(m,\delta) \leq A(\delta)^{r/(r-1)} \|f\|_{p_0}^p,
\end{equation}
where $A(\delta) $ is the $n$-dimensional surface area of the geodesic ball $B(a^*, \delta)$.
For the second integral in (\ref{E:I12}) we have
\begin{equation}\label{E:I2}
I_2(m,\delta) <  \sigma_n \sup\limits_{x \in K_{\delta}}|(W^m f)(x)|^p,
\end{equation}
where $\sigma_n$ is the area of the unit sphere $\sn$.

Now we fix sufficiently small $\varepsilon >0$. Using (\ref{E:I1}), let us  choose $\delta > 0$ so  that $I_1(m,\delta) < \e/2$ for all $m\ge 0$.
 By Theorem \ref{L:uniformly}, the inequality (\ref{E:I2}) implies that
 there exists
 $\tilde m =\tilde m (\del)$ such that $I_2(m,\delta) < \e/2$ for all $m \ge \tilde m$. Hence, by (\ref{E:I12}),
$\|W^mf \|^p_{p} < \varepsilon$ for $m \geq \tilde m$, and therefore $W^mf$ tends to $0$  as $ m \to \infty$ in the $L^p$-norm. The latter gives the desired convergence of the series (\ref{bbi}).

  On the other hand, if $p > p_0$, then,  by  H\"older's inequality, $\|f\|_{p_0} =\|W^m f\|_{p_0} \leq c \,\|W^mf\|_{p}$, $c=\const >0$.
 It follows that the $L^p$-norm of the remainder $W^m f$ of the series does not tend to $0$ as $ m \to \infty$,  unless $f=0$.

 The proof for $\tilde W f$ is similar.
\end{proof}

\begin{remark}
 As we can see, the iterative method  in terms of the  series (\ref{bbi}) and (\ref{bbi2})  does not provide  uniformly convergent reconstruction of continuous functions. The  reconstruction is guaranteed only in the  $L^p$-norm with $1 \leq  p < p_0=n/(k-1)$. For instance, in the case of the hyperplane sections,   when  $k=n$ and $p_0=1+1/(n-1)$, the $L^2$-convergence fails because $p_0$ does not exceed $2$. The  less  the dimension  $k$ is, the greater  exponent $p$  can be chosen.
 The case $p=1$ works for all $1<k\le n$.
\end{remark}

\section{ Proof of Theorem \ref {kdndt} }\label {jsbsr}

 We recall that $a\in \bbb^{n+1}$, $s_a=\sqrt{1-|a|^2}$,  and $\vp_a$  is an automorphism (\ref{E:g_a}).
The following lemma  allows us to exploit the language of Stiefel manifolds when dealing with affine planes.

\begin{lemma} \label{lemuupz}  Let  $1\le k \le n$.
The  map $\vp_a$ extends  as a bijection from  ${\rm Gr}_a (n+1, k)$ onto   ${\rm Gr}_o (n+1, k)$. Specifically, if $\t  \in {\rm Gr}_a (n+1, k)$ is defined by
\be\label{kuytrs} \t=\{ x\in \bbr^{n+1} :\, \xi'x =\xi'a\}, \qquad \xi\in  \stnk,\ee
 then $\z\equiv \vp_a \t  \in {\rm Gr}_o (n+1, k)$ has the form
  \be\label{kuy44trs} \z=\{ y\in \bbr^{n+1} :\, \eta'y =0\}, \qquad \eta\in  \stnk,\ee
where
\be\label {obbmkas} \eta = -(\mathrm {A} \xi) \, \a^{-1/2}, \qquad \mathrm {A}=s_a \mathrm {P}_a  +  \mathrm {Q}_a, \qquad \a= (\mathrm {A} \xi)'(\mathrm {A} \xi).\ee

Conversely, if $\z \in {\rm Gr}_o (n+1, k)$ is defined by (\ref{kuy44trs}), then $\t\equiv\vp_a \z$ has the form (\ref{kuytrs}) with
\be\label {obbmkb} \xi = (\mathrm {A}_1\eta) \, \b^{-1/2}, \qquad \mathrm {A}_1=\mathrm {P}_a  + s_a  \mathrm {Q}_a, \qquad \, \b=(\mathrm {A}_1\eta)'(\mathrm {A}_1\eta).
 \ee
\end{lemma}
\begin{proof} Let $\t\in {\rm Gr}_a (n+1, k)$ be defined by  (\ref{kuytrs}).
 Then
\[
\z\equiv\vp_a \t=\{ y \in \bbr^{n+1} :\, \xi'(\vp_a y -a)=0 \}.\]
By (\ref{E:g_a}),
\be\label {irrr} \vp_a y -a=\frac{s_a \mathrm {A}y}{1-y\cdot  a}.\ee
 Hence $ \z=\{ y \in \bbr^{n+1} :\, (\mathrm {A}\xi)'y =0 \}$.
Now  (\ref{kuy44trs}) follows if we
 represent the $(n+1)\times (n+1-k)$ matrix  $\mathrm {A} \xi$  in the polar form
 \be\label{lzxg}
 \mathrm {A} \xi=\eta \, \a^{1/2}, \quad \, \a= (\mathrm {A} \xi)'(\mathrm {A} \xi), \quad \eta = (\mathrm {A} \xi) \, \a^{-1/2};\ee
 see, e.g., \cite[pp. 66, 591]{Mu}.

Conversely, let  $\z\in {\rm Gr}_o (n+1, k)$ be defined by  (\ref{kuy44trs}).
Then \[\t\equiv\vp_a \z=\{x \in \bbr^{n+1} :\, \eta' \vp_a x =0 \}.\]
By (\ref{E:g_a}), the equality  $\eta' \vp_a x =0$ is equivalent to
\be\label {lanrtb}
(\mathrm {P}_a \eta + s_a  \mathrm {Q}_a\eta)'x=\eta'a \quad \text{\rm or}\quad  (\mathrm {A}_1\eta)'x=\eta'a .\ee
We write  $\mathrm {A}_1\eta$ in the  form $\mathrm {A}_1\eta =\xi \, \b^{1/2}$ with  $\, \b=(\mathrm {A}_1\eta)'(\mathrm {A}_1\eta)$ and $\xi = (\mathrm {A}_1\eta) \, \b^{-1/2} \in   \stnk$. Then
(\ref{lanrtb}) yields $\xi'x= \, \b^{-1/2}\eta'a$. To complete the proof, it remains to note that $\, \b^{-1/2}\eta'a=\xi'a$. Indeed,
\bea
\xi'a&=&\, \b^{-1/2}(\mathrm {A}_1\eta)'a=\, \b^{-1/2}(\mathrm {P}_a \eta + s_a  \mathrm {Q}_a\eta)'a\nonumber\\
&=& \, \b^{-1/2}(\eta' (\mathrm {P}_a a  +s_a  \eta'\mathrm {Q}_a a)= \, \b^{-1/2}\eta' a.\nonumber\eea
\end{proof}

\noindent {\bf Proof of the Theorem.}
The case $k=1$  is almost obvious; cf. Remark \ref{klut}.   Assuming $1<k\le n$, let $\t  \in {\rm Gr}_a (n+1, k)$ have the form
(\ref{kuytrs})  and write
 \[(F_a f)(\t)\equiv (F_a f)(\xi)\! = \!\!\intl_{\{x\in \sn: \xi' (x-a)=0\}} \!\!\!f(x)\, d\sig (x), \quad \xi\in  \stnk.\]
  We make use of the standard approximation machinery.  Given  a sufficiently small  $\e >0$, let
\be\label{kuyb}
 (F_{a, \e} f)(\xi)=\intl_{\sn} f(x)\, \om_\e (\xi'(x-a))\, dx,\ee
where $\om_\e$ is a smooth bump function supported on the ball  in $\bbr^{n+1-k}$ of radius $\e$ with center at the origin, so that
$\lim\limits_{\e\to 0} \int_{|t|<\e}\om_\e (t)\,g(t)\,dx=g(o)$
 for any function $g$ which is continuous in a neighborhood of the origin.

STEP I. Let us show that
\be\label{kuyb1}
 \lim\limits_{\e\to 0} (F_{a, \e} f)(\xi)=(1-|\xi' a|^2)^{-1/2}(F_{a} f)(\xi).\ee

We pass to bispherical coordinates   (see, e.g., \cite [p. 31]{Ru15})
\be\label{niys} x= \left[\begin{array} {c}
\vp \, \sin \th \\ \psi\, \cos \,\th \end{array} \right], \quad \vp \in \sn \cap \xi^\perp, \quad \psi \in \sn \cap \{\xi\}, \quad 0 \!\le\!\theta \!\le\! \pi/2,\ee
\[
 \qquad  dx=
 \sin^{k-1}\theta \, \cos^{n-k}\theta \, d\theta d\vp d\psi,\]
and set $s=\cos\, \th$.
This gives
\bea
&&(F_{a, \e} f)(\xi)= \intl_{0}^1 s^{n-k} (1-s^2)^{(k-2)/2} ds \intl_{\sn \cap \xi^\perp} d\vp \nonumber\\
&&\times \intl_{\sn \cap \{\xi\}}
f \left(\left[\begin{array} {c}
\vp \, \sqrt{1-s^2} \\ s\psi\ \end{array} \right] \right )\, \om_\e (s\psi -\xi'a)\, d\psi\nonumber\\
\label{Riesz} &&= \intl_{\bbr^{n-k+1}} H(y) \, \om_\e (y - \xi'a) \,dy,  \eea
where
\[H(y)=(1\!-\!|y|^2)^{(k-2)/2} \intl_{\sn \cap \xi^\perp} f \left(\left[\begin{array} {c}
\vp \, \sqrt{1\!-\!|y|^2} \\ y\end{array} \right] \right )d\vp\]
if $|y|\le 1$ and $H(y)=0$, otherwise. Passing to the limit, we obtain
\[ \lim\limits_{\e\to 0} (F_{a, \e} f)(\xi)=
 H(\xi'a),\]
 where
\be\label{ryt}  H(\xi'a)=(1\!-\!|\xi'a|^2)^{(k-2)/2} \intl_{\sn \cap \xi^\perp} f \left(\left[\begin{array} {c}
\vp \, \sqrt{1\!-\!|\xi'a|^2} \\ \xi'a\end{array} \right] \right )d\vp. \ee
 If the argument of $f$  is denoted by $x$,  then $x -a$ lies in the subspace perpendicular to $\xi$.
Further, the integration in (\ref{ryt}) is performed over the $(k-1)$-dimensional sphere of radius $\sqrt{1\!-\!|\xi'a|^2}$. Switching to the surface area measure, we can write (\ref{ryt}) as
\[
 H(\xi'a)=(1\!-\!|\xi'a|^2)^{-1/2} \intl_{\{x \in \sn: \,\xi'(x -a)=0\}} f(x)\, d\sig (x),\]
as desired.

STEP II. Let us obtain an alternative expression for the limit
(\ref{kuyb1}), now in terms of the automorphism $\vp_a$. By Lemma \ref{lem31},
\[
 (F_{a, \e} f)(\xi)= s_a^n \intl_{\sn} \frac{(f\circ \vp_a)(y)}
{(1-  a \cdot y)^n}\, \om_\e (\xi'[\vp_a y-a])\, dy,\]
where
\[
\xi'[\vp_a y-a]= -\frac{s_a\xi'\mathrm {A}y}{1-  a \cdot y}=
-\frac{s_a (\mathrm {A}\xi)'y}{1-  a \cdot y}, \qquad \mathrm {A}=s_a \mathrm {P}_a +  \mathrm {Q}_a;\]
see (\ref{irrr}). Denote
\[ \tilde f (y)= s_a^n \frac{(f\circ \vp_a)(y)}
{(1-  a \cdot y)^n}.\]
Then
\[
 (F_{a, \e} f)(\xi)= \intl_{\sn} \tilde f (y)\, \om_\e \left(\frac{s_a(\mathrm {A} \xi)'y}{1-a\cdot  y}\right ) dy.\]
As in (\ref{lzxg}),  the polar decomposition yields
\be\label {jnlli} \mathrm {A} \xi=\eta\, \a^{1/2}, \quad \, \a= (\mathrm {A} \xi)'(\mathrm {A} \xi), \quad \eta = (\mathrm {A} \xi)  \, \a^{-1/2}.\ee

Then we pass to bispherical coordinates   (cf. (\ref{niys}))
\[ y= \left[\begin{array} {c}
\vp \, \sin \th \\ \psi\, \cos \,\th \end{array} \right], \quad   \vp \in \sn \cap \eta^\perp, \quad \psi \in \sn \cap \{\eta\}, \quad 0 \!\le\!\theta \!\le\! \pi/2,,\]
\[
 \qquad  dy=
 \sin^{k-1}\theta \, \cos^{n-k}\theta \, d\theta d\vp d\psi,\]
and set $s=\cos\, \th$.
This gives
\bea
&&(F_{a, \e} f)(\xi)= \intl_{0}^1 s^{n-k} (1-s^2)^{(k-2)/2} ds \intl_{\sn \cap \eta^\perp} d\vp \nonumber\\
&&\times \intl_{\sn \cap\{\eta\}}
\tilde f \left(\left[\begin{array} {c}
\sqrt{1-s^2} \vp   \\ s\psi\ \end{array} \right] \right )\, \om_\e \left(\frac{s_a\, \a^{1/2} s\psi}{1-a \cdot  (\sqrt{1-s^2}\, \vp+s\psi)}\right ) \, d\psi,\nonumber \eea

or (set $z=s\psi \in \{\eta\} \sim \bbr^{n+1-k}$, $|z|<1$)
\bea
&&(F_{a, \e} f)(\xi)= \intl_{|z|<1} (1-|z|^2)^{(k-2)/2} dz  \nonumber\\
&&\times
\intl_{\sn \cap \eta^\perp} \tilde f \left(\left[\begin{array} {c}
 \sqrt{1-|z|^2} \vp  \\ z\ \end{array} \right] \right )\, \om_\e \left(\frac{s_a\, \a^{1/2} z}{1-a \cdot (\sqrt{1-|z|^2}\, \vp+z)}\right ) \, d\vp,\nonumber \eea
We set
\be\label{kabe}
t\equiv t(z)=\frac{s_a\, \a^{1/2} z}{1-a \cdot (\sqrt{1-|z|^2}\, \vp+z)}=\frac{\Lam z}{1-h(z)},\ee
\[ \Lam =s_a\, \a^{1/2}, \qquad h(z)=a \cdot (\sqrt{1-|z|^2}\, \vp+z),\]
 so that $t=o$ if and only if $z=o$, where $o$ is the origin in the corresponding space. Further, we write (\ref{kabe}) as
\[ \Phi (t,z)\equiv \Lam z -t+t h(z)=0\]
and denote $m=n+1-k$.
Because the $m\times m$ matrix $(\partial \Phi_i/\partial z_j)(o,o)=\Lam$ is invertible, there exists an inverse function $z=z(t)$, which is  well-defined and differentiable in a  small neighborhood of $t=0$. Hence, for sufficiently small $\e>0$,
\bea
&&(F_{a, \e} f)(\xi)= \intl_{|t|<\e} (1-|z(t)|^2)^{(k-2)/2} \om_\e (t) \, |\det (z'(t))|\,dt  \nonumber\\
&&\times
\intl_{\sn \cap \eta^\perp} \tilde f \left(\left[\begin{array} {c}
 \sqrt{1-|z(t)|^2} \vp  \\ z(t)\ \end{array} \right] \right )\, \, d\vp,\nonumber \eea
where
\[z' (t)=- \left[\frac{\partial \Phi (t,z)}{ \partial z}\right ]^{-1} \, \frac{\partial \Phi (t,z)}{ \partial t}, \qquad z=z(t).\]
 Passing to the limit, we obtain
\[
\lim\limits_{\e\to 0} (F_{a, \e} f)(\xi)\!=\!|\det (z'(o))| \intl_{\sn \cap \eta^\perp} \tilde f (\vp)\, \, d\vp,\]
where
\[ z'(o)= (1-a\cdot\vp) \Lam^{-1}= s_a^{-1}(1-a\cdot\vp)\, \a^{-1/2}, \qquad \, \a= (\mathrm {A} \xi)'(\mathrm {A} \xi).\]
This gives
\[
\lim\limits_{\e\to 0} (F_{a, \e} f)(\xi)\!=\!\frac{s_a^{k-1}}{\det (\a)^{1/2}} \intl_{\sn \cap \eta^\perp} \frac{(f\circ \vp_a)(\vp)}
{(1-  a \cdot \vp)^{k-1}}\, d\vp.\]
Note that
\bea \, \a&=& (\mathrm {A} \xi)'(\mathrm {A} \xi)= (s_a \mathrm {P}_a \xi+  \mathrm {Q}_a\xi)'(s_a \mathrm {P}_a \xi+  \mathrm {Q}_a\xi)\nonumber\\
&=&I_{n+1-k}- \xi'aa'\xi,\nonumber\eea
and therefore $\det (\a)=\det (I_{n+1-k}- \xi'aa'\xi)$.
The last expression can be transformed by making use  of
  the known fact  from Algebra (see, e.g., \cite [Theorem A3.5]{Mu}). Specifically,
  if $\mathrm {U}$ and $\mathrm {V}$ are  $m\times n$ and $n \times m$ matrices, respectively, then
\be\label{aooi4}
\det (\mathrm {I}_{m} +\mathrm {U}\mathrm {V}) =\det (\mathrm {I}_{n} +\mathrm {V}\mathrm {U}).\ee
By this formula, $\det (\a)=1-(a'\xi)(\xi'a)=1-|\xi'a|^2$. Thus, changing  notation, as in (\ref{vt1de}), we have
\be\label{aooi41}
\lim\limits_{\e\to 0} (F_{a, \e} f)(\xi)=(1-|\xi'a|^2)^{-1/2} \intl_{\sn \cap \eta^\perp} (M_a f)(y)\, d\sig (y), \ee
where $\eta = (\mathrm {A} \xi)  \, \a^{-1/2}$; cf. (\ref{jnlli}).

STEP III. Comparing (\ref{kuyb1}) with (\ref{aooi41}) and  switching backward to the Grassmannian language (use Lemma \ref{lemuupz}), we obtain the statement of the theorem.

\bibliographystyle{amsplain}

\end{document}